\documentclass[12pt,reqno]{amsart}
\usepackage[left=3cm,top=2cm,right=3cm,bottom=2cm]{geometry}
\usepackage{amssymb}
\usepackage{amsbsy}
\usepackage{epsfig}
\usepackage{tikz}
\usepackage{verbatim}
\usepackage[pdftex]{hyperref}
\hypersetup{colorlinks=true,linkcolor=blue,citecolor=blue,breaklinks = true}

\tikzstyle{vertex} = [fill,shape=circle,node distance=80pt]
\tikzstyle{edge} = [fill,opacity=.5,fill opacity=.5,line cap=round, line join=round, line width=50pt]
\tikzstyle{elabel} =  [fill,shape=circle,node distance=30pt]

\pgfdeclarelayer{background}
\pgfsetlayers{background,main}

\begin{document}
	\title{Principal eigenvectors of general hypergraphs}
	\author[K. Cardoso]{Kau\^e Cardoso} \address{Instituto Federal do Rio Grande do Sul - Campus Feliz, CEP 95770-000, Feliz, RS, Brasil} \email{\tt
		kaue.cardoso@feliz.ifrs.edu.br}
	
	\author[V.Trevisan]{Vilmar Trevisan} \address{Instituto de Matem\'atica e Estat\'{\i}stica, UFRGS,  CEP 91509--900, Porto Alegre, RS, Brazil}
	\email{\tt trevisan@mat.ufrgs.br}
	\date{May 4, 2019}

	\pdfpagewidth 8.5 in \pdfpageheight 11 in

\newcommand{\h}{\mathcal{H}}
\newcommand{\hk}{\mathcal{H}^k}
\newcommand{\A}{\mathbf{A}}
\newcommand{\x}{\mathbf{x}}
\newcommand{\y}{\mathbf{y}}
\newcommand{\Ah}{\mathbf{A}_\mathcal{H}}
\newcommand{\Ahk}{\mathbf{A}_{\mathcal{H}^k}}
\newcommand{\Dh}{\mathbf{D}_\mathcal{H}}
\newcommand{\Lh}{\mathbf{L}_\mathcal{H}}
\newcommand{\Lhk}{\mathbf{L}_{\mathcal{H}^k}}
\newcommand{\Qh}{\mathbf{Q}_\mathcal{H}}
\newcommand{\Qhk}{\mathbf{Q}_{\mathcal{H}^k}}
\newcommand{\C}{\mathbb{C}}
\newcommand{\Dem}{\textbf{Proof.} \newline }
\newcommand{\cvt}[1]{\textcolor{green}{#1}}
\newcommand{\ckc}[1]{\textcolor{red}{#1}}
	\newtheorem{Pro}{Proposition}
	\newtheorem{Cor}{Corollary}
	\newtheorem{Lem}{Lemma}
	\newtheorem{Def}{Definition}
	\newtheorem{Teo}{Theorem}
	\newtheorem{Exe}{Exemple}
	\newtheorem{Obs}{Remark}

\maketitle
\begin{abstract}
In this paper we obtain bounds for the extreme entries of the principal
eigenvector of  hypergraphs; these bounds are computed using the spectral
radius and some classical parameters such as maximum and minimum degrees.
We also study  inequalities involving the ratio and difference between the
two extreme entries of this vector.\newline


\end{abstract}

\section{Introduction}

One of the main goals of spectral graph theory is to understand the structure
of graphs through its associated matrices and their spectra.  For many years,
researchers around the world have tried to develop a similar theory for
hypergraphs. Several proposals were made between the late 1990s and the early 2000s, and we believe that the main representative attempts presented in
this period were given in \cite{Feng, Friedman, Rodriguez1}. None of these
proposals seemed to be widely accepted or carried over by the community
of researchers in this area. This scenario changed in 2012, with the work
\cite{Cooper} presented by Cooper and Dutle. In that paper, the concept of an
adjacency tensor of a uniform hypergraph was introduced, based on the
developments of the spectral tensor theory, which begun in 2005 by Qi in
\cite{Qi-2005} (for definitions read section 2). Cooper and Dutle generalized
many important results of the spectral graph theory, initiating the spectral
hypergraph theory via tensors.

Most of the current literature in this area is devoted to uniform hypergraphs. 
While the spectral theory of uniform hypergraphs already has some
recognition, studies of general hypergraphs are still taking their first
steps. Our main concern in the present paper is to the study spectral
properties of general hypergraphs. To achieve this, we will use the
definition of adjacency tensor presented by Benerjee, Char and Mondal in
\cite{Banerjee}. It is worth mentioning other published papers on the spectra of general hypergraphs
such as \cite{Zhang, Kang}. Finally, the work \cite{Bu} proposes a definition of an
adjacency tensor different from that proposed in \cite{Banerjee}, that we use here.

Our paper is anchored in a version of Perron-Frobenius Theorem for
hypergraphs, presented below.

\begin{Teo}[Theorem 3.1, \cite{Zhang}]\label{Perron}		
	Let $\h$ be a hypergraph and $\rho$ its spectral radius, then
	\begin{itemize}
		\item[(a)] $\rho$ is an eigenvalue of $\h$, with a non-negative real eigenvector.
		\item[(b)] If $\h$ is conected then $\rho$ is the unique eigenvalue of $\h$ with a strictly positive eigenvector, and this eigenvector is unique (up to a positive multiplier).
	\end{itemize}
\end{Teo}

We will call \emph{principal eigenvector} of a conected hypergraph, the
vector obtained in part (b) of the Perron-Frobenius Theorem, normalized to
norm $\ell_k$.  The main object of study of this note is this vector, more
specifically, we are interested in studying the extreme entries - the largest
and the smallest entries - of the principal eigenvector of a conected
hypergraph.

The study of the principal eigenvector for hypergraphs is interesting because the
value of each of its entry may be seen as a spectral measure of the
centrality of the vertex associated with this entry. Another interesting
property of this vector is that the quotient and the subtraction of its two
extreme entries can be understood as measurements of the irregularity of the
hypergraph. In addition, an important optimization problem for graphs is to
determine which vertex or edge that, when removed from a graph, causes the
greatest decrease in the spectral radius. The answer for this problem, given
in \cite{Stevanovic}, is to remove a vertex with maximum entry in the main
eigenvector $\x = (x_v)$, or remove the edge $e = uv$ such that $x^e =
x_ux_v$ has maximum value. We see that many important information from a
graph can be obtained through the study of the principal eigenvector. In
fact, bounding the entries of the main eigenvector of a graph is a topic of
many research works. In particular, Stevanovic has a chapter in his book
\cite{Stevanovic} dedicated to the study of these parameters. The reader is
referred to the this chapter and references therein.

For tensors and hypergraphs these entries have been studied in
the papers \cite{LiuQ, Li, LiuL, Si, Kang}. In this paper we obtain bounds for the
largest and smallest entries of the principal eigenvector of a general
hypergraph. Many of our results are inspired by results proven for graphs by
Cioab\u{a} and Gregory in \cite{Cioaba}. The bounds presented throughout this
work are always best possible for regular hypergraphs.  We emphasize the
result presented in Theorem \ref{teoxmin1} since it is sharp for a larger
class of hypergraphs.

\begin{Teo}\label{teoxmin1}
Let $ \h $ be a connected hypergraph on $n$ vertices and rank $k$. If $(\rho,
\x)$ is its	principal eigenpair, then
	\[x_{min} \leq \sqrt[k]{\frac{\delta}{\rho+\delta(n-1)}}\]

Equality holds if and only if there is a vertex $v$ such that, for all $w
	\in V-\{v\}$, we have  $x_w = x_{min}$ and $x_v =
	\sqrt[k]{\frac{\rho}{\delta}}x_{min}$.
\end{Teo}

We call the attention for the last section of this note, because we have
defined and studied some parameters that have not yet been properly explored.
When $\h$ is a uniform hypergraph, we define, for each edge  $e=\{v_1,\ldots,v_k\}$, the number $x^e = x_{v_1}\cdots x_{v_k}$, where $\x=(x_v)$ is the principal eigenvector of this hypergraph. We will say that $x^e$  is the value of the edge $e$  in the principal eigenvector. Thus we define the parameters  $x^{max}$ and $x^{min}$ as the largest and the smallest value reached by $x^e$ and the parameter $\Gamma(\h)$ as the quotient between $x^{max}$ and $x^{min}$.
Theorem \ref{teo:caracG} is the main result of this part of the work.

\begin{Teo}\label{teo:caracG}
	Let $\h$ a $k$-graph. $\Gamma(\h) = 1$, if and only if,
	for each edge the product of the degrees of its vertices is constant.
\end{Teo}

We notice that if a hypergraph $\h$ has the parameter $\Gamma(\h)$ greater than 1,
then the product of the vertices in each edge is not constant, so we can say
that $\Gamma(\h)$ is a measure of the distribution of degrees of
vertices along the edges of the hypergraph.

The remaining of the paper is organized as follows. In Section \ref{sec:pre}
we present some basic definitions about hypergraphs and tensors. In
Section \ref{sec:measure} we obtain bounds for the difference and for the
ratio between the two extreme entries of the principal eigenvector. In
Section \ref{sec:bounds} we will prove some relations between the extreme
entries of the principal eigenvector with other important parameters. In Section \ref{sec:edges} we will prove Theorem \ref{teo:caracG}, and we will construct some inequations involving the parameters $x^{max}$, $x^{min}$ and $\Gamma$.

\section{Preliminaries}\label{sec:pre}
In this section, we shall present some basic definitions about hypergraphs
and tensors, as well as terminology, notation and concepts that will be
useful in our proofs. More details can be found in  \cite{Banerjee, Cooper,
LiuL, Shao}.

\begin{Def}
	A tensor (or hypermatrix) $\A$ of dimension $n$ and order $r$ is a collection of $n^r$
	elements $a_{i_1 \dots i_r}\in\mathbb{C}$ where
	$i_1,\dots,i_r\in[n] =\{1,2,\ldots,n\}$.
\end{Def}

Let $\A$ be a tensor of dimension $n$ and order $r$ and  $\mathbf{B}$ be a
tensor of dimension $n$ and order $s$. We define the product of  $\A$ by
$\mathbf{B}$ as a tensor $\mathbf{C}=\A\mathbf{B}$ of dimension $n$ and
order $(r-1)(s-1)+1$, where	
\[c_{j\alpha_1\cdots\alpha_{r-1}}=\sum_{i_2,\dots, i_r = 1}^{n}a_{ji_2\dots i_r}b_{i_2\alpha_1}\cdots b_{i_r\alpha_{r-1}}\;\; \emph{with} \;\; j\in[n]\;\; \emph{and} \;\;\alpha_1,\dots,\alpha_{r-1}\in[n]^{s-1}.\]

In particular, if $\x\in \mathbb{C}^n$ is a vector, then the $i$th component of the product $\A \x$
is given by
\[(\A \x)_i=\sum_{i_2,\dots, i_r = 1}^{n}a_{ii_2\dots i_r}x_{i_2}\cdots x_{i_r}\quad\forall i\in [n].\]

If $\x=(x_1,\cdots,x_n)\in \mathbb{C}^n$ is a vector and $r$ is a positive integer, we let $\x^{[r]}$ denote the vector in $\mathbb{C}^n$ whose $i$th component is given by $x_i^r$.
\begin{Def}
	A number $\lambda \in \mathbb{C}$ is an eigenvalue of a tensor $\A$ of dimension $n$ and order $r$ if
	there is a nonzero vector $\x\in \mathbb{C}^n$ such that
	\[\A \x=\lambda \x^{[r-1]}.\]
	We say that $\x$ is an eigenvector of $\A$ associated with the eigenvalue $\lambda$ and that
	$(\lambda, \x)$ is an eigenpair of $\A$.
\end{Def}
An eigenpair $(\lambda,\x)$ is \emph{strictly positive} if the eigenvalue $\lambda > 0$ and all the entries of the eigenvector $\x$ are positive.

\begin{Def}	
The spectral radius of a tensor $\A$  of order $k$, is the
largest modulus of an eigenvalue of $\A$, that is
	\[\rho(\A) = \max\{|\lambda|: \A\x = \lambda \x^{k-1}, \textrm{ for } \x\neq0\} \]
\end{Def}

\begin{Def}
A hypergraph $\h=(V,E)$ is a pair formed by a set of vertices $V(\h)$ and a set of edges $E(\h)\subseteq 2^V$, where $2^V$ is the power set of $V$.
\end{Def}
$\h$ is said to be a $k$-uniform (or a $k$-graph) for an integer $k \geq 2$, if all the edges of $\h$ have cardinality $k$.

\begin{Def}
Let $ \h $ be a hypergraph. The neighborhood of a vertex $v$, denoted by $N(v)$, is the set formed by all distinct vertices of $v$, that have some common edge with $v$.
\end{Def}

\begin{Def} Let $ \h $ be a hypergraph. We define the following sets
\[E_{(i)}=\{e-i:i\in e\in E(\h)\}, \quad E_{[i]}=\{e:i\in e\in E(\h)\}\]
\end{Def}
		
\begin{Def}
Let $\h = (V,E)$ be a hypergraph. A path is a sequence of vertices and edges $v_0e_1v_1e_2 \ldots e_lv_l$ where $v_{i-1}$ and $v_i$ are contained in $e_i$ for each $ i \in [l] $ and the vertices $ v_0, v_1, \ldots v_l $ are all distinct, as well as the edges $ e_1, \cdots, e_l $ are also all distinct. In these conditions we say that the hypergraph $\h$ is connected, if for each pair of vertices $ u, w \in V $ there is a path $ v_0e_1v_1e_2 \cdots e_lv_l $ where $ u = v_0 $ and $ w = v_l $. Otherwise $ \h $ is said to be disconnected.	
\end{Def}

\begin{Def}
Let $\h=(V,E)$ be a hypergraph. The degree of a vertex $v\in V$, denoted by
$d(v)$, is the number of edges that contain $v$.
\end{Def}	
A hypergraph is $r$-regular if $d(v) = r$ for all $v \in V$. We will also define the maximum, minimum and average degrees as follows
\[\Delta(\h) = \max_{v \in V}\{d(v)\}, \quad \delta(\h) = \min_{v \in V}\{d(v)\}, \quad d(\h) = \frac{1}{n}\sum_{v \in V}d(v)\]

\begin{Def}
Let $\h$ be a hypergraph. The rank of $\h$ is the maximum cardinality of the
edges in the hypergraph, let us denote this parameter by $r(\h)$.
\end{Def}
Observe that, if a hypergraph is $k$-uniform then its rank is $k$, because all its edges have size $k$. From this point on, $k$ represents the rank of the hypergraph being treated.

\begin{Def}
	Let $\h=(V,E)$ be a hypergraph.
	\begin{enumerate}
\item For each edge $e \in E$, we say that an ordered sequence
    $\alpha=(v_1,v_2,\ldots,v_k)$ is an $k$-expanded edge from $e$, if its set
    of the distinct elements is equal to edge $e$, and denote this by $e \prec
    \alpha$.
		
\item For each edge $e \in E$, we define $S(e) = \{\alpha \;| \;e \prec
    \alpha\}$, the set of all $k$-expanded edges from $e$.
		
\item For each edge $e \in E$ and vertex $v \in e$, we define $S(e)_v = \{
    \alpha \in S(e) \;|\; \alpha=(v,v_2\cdots,v_k)\}$,	the set of all
    $k$-expanded  edges in $S(e)$, starting with $v$.
		
\item We define $S(\h) = \displaystyle\bigcup_{e \in E}S(e)$,  the set of all
    $k$-expanded  edges.
		
\item  For each vertex $v \in V$, we define $S(\h)_v = \displaystyle\bigcup_{e
    \in E_{[v]}}S(e)_v$, the set of all  $k$-expanded edges in $S(\h)$, starting with
    $v$.
	\end{enumerate}
\end{Def}

\begin{Obs}\label{lemaS}
Let $\h$ be a hypergraph, For each edge $e \in E$ it holds that
	\begin{enumerate}
\item If $|e| = r$ then $|S(e)| = \displaystyle\sum_{s_1 + \cdots + s_{r} = k}\frac{k!}{s_1!\cdots s_r!}$, with $s_i\geq 1$ for $i \in [r]$.
\vspace{0.2cm}		
\item If $v \in e$ then $|S(e)| = |e||S(e)_v|$
	\end{enumerate}	
\end{Obs}

\begin{Def}
Let $\h$ be a hypergraph on $n$ vertices, we define the adjacency tensor
$A_\h$, as a tensor of dimension $n$ and order $k$ where
	\[ a_{i_1 \dots i_k}= \begin{cases}
	\frac{|e|}{|S(e)|}\; \textrm{if } e \prec (i_1, \dots, i_k) \textrm{ for some } e \in E(\h)\\
	\;\;\; 0 \;\;\; \textrm{otherwise} \end{cases}\]
\end{Def}

For each edge $e \in E(\h)$ denote $a(e) = \frac{|e|}{|S(e)|}$. For each $k$-expanded
edge $\alpha = (i_1,\cdots,i_k) \in S(\h)$ denote $a_\alpha = a_{i_1i_2\cdots
i_k}$, $x^\alpha = x_{i_1}\cdots x_{i_k}$ and $x^{\alpha - i_l}=
x_{i_1}\cdots x_{i_{l-1}}x_{i_{l+1}}\cdots x_{i_k}$. With these
considerations we have

\[(\Ah\x)_i=\sum_{i_2,\dots, i_k = 1}^{n}a_{ii_2\cdots
i_k}x_{i_2}\cdots x_{i_k} = \sum_{\alpha\in S(\h)_i}a_\alpha x^{\alpha-i} =
\sum_{e\in E_{[i]}}a(e)\sum_{\alpha \in S(e)_i}x^{\alpha-i}\]

Therefore to determine the eigenvalues of $\h$ we need to solve the following
system
\begin{equation}\label{eq:sistema}
 \sum_{e\in E_{[i]}}a(e)\sum_{\alpha \in S(e)_i}x^{\alpha-i}=\lambda x_i^{k-1}\quad \forall i\in V(\h)
\end{equation}

Another interesting formula is
\begin{equation}\label{eq:xax}
\x^T\Ah\x = \sum_{i \in V}
\left(x_i\sum_{e\in E_{[i]}}a(e)\sum_{\alpha \in S(e)_i}x^{\alpha-i}\right) =
\sum_{e \in E}a(e)\sum_{\alpha \in S(e)}x^{\alpha}
\end{equation}

For uniform hypergraphs the formulas \ref{eq:sistema} and \ref{eq:xax} are reduced to

 \[(\Ah\x)_i = \sum_{e\in E_{[i]}}x^{e-i}, \quad \x^T\Ah\x = k\sum_{e \in E}x^e\]

\begin{Def}
Let $\h$ be a connected hypergraph on $n$ vertices, and $\x$ the positive
eigenvector associated to $\rho(\h)$ obtained in Theorem \ref{Perron} with
$||\x||_k=1$. We will call $(\rho,\x)$ of principal eigenpair and $\x$ of
principal eigenvector.
\end{Def}

\begin{Def}	
	Let $\h$ be a hypergraph and $\x=(x_v)$ its principal eigenvector.
	\begin{enumerate}
		\item $\quad \displaystyle x_{min} = \min_{ v \in V(\h)}\{x_v\}$
		\vspace{0.2cm}
		\item $\quad \displaystyle x_{max} = \max_{ v \in V(\h)}\{x_v\}$
		\vspace{0.2cm}
		\item $\quad \displaystyle\sigma(\h) = x_{max} - x_{min}$
		\vspace{0.2cm}
		\item $\quad \displaystyle\gamma(\h) = \frac{x_{max}}{x_{min}}$
	\end{enumerate}
\end{Def}

\section{Irregularity measurements of hypergraphs}\label{sec:measure}

In this section we will study the parameters $\sigma(\h)$ and $\gamma(\h)$,
they can be used as measurements of the irregularity of the hypergraph $\h$,
this occurs because $\h$ is regular only when its principal eigenvector has
all coordinates equals. More precisily, we will prove some results that
help to estimate the value of these parameters.

\begin{Teo}\label{teo:reg}
	A hypergraph $\h$ is $r$-regular, if and only if, $(r,\x)$ is its
principal eigenpair, with $\x = (\frac{1}{\sqrt[k]{n}},\ldots,\frac{1}{\sqrt[k]{n}})$.
\end{Teo}
\begin{proof} If $\h$ is $r$-regular then
\[(\Ah\x)_i = \sum_{e\in
E_{[i]}}a(e)\!\!\!\sum_{\alpha \in S(e)_i}\left(\frac{1}{\sqrt[k]{n}}
\right)^{k-1}\!\!\!\!=\sum_{e\in
E_{[i]}}\underbrace{a(e)|S(e)_i|}_{=1}\left(\frac{1}{\sqrt[k]{n}}
\right)^{k-1}\!\!\!\!= d(i)\left(\frac{1}{\sqrt[k]{n}} \right)^{k-1}\!\!\!\!
= rx_i^{k-1}\] Therefore $(r,\x)$ is an eigenpair of $\h$ and by Theorem
\ref{Perron} we conclude that this is the principal eigenpair.

Conversely, if $(r,\x)$ is the principal eigenpair of $\h$, then for each $u
\in V$ we have
\[r\left( \frac{1}{\sqrt[k]{n}}\right)^{k-1}= \sum_{e\in
E_{[u]}}a(e)\sum_{\alpha \in S(e)_u}\left( \frac{1}{\sqrt[k]{n}}\right)^{k-1} =
\sum_{e\in E_{[u]}}\underbrace{a(e)|S(e)_u|}_{=1}\left(\frac{1}{\sqrt[k]{n}}
\right)^{k-1}   = d(u)\left( \frac{1}{\sqrt[k]{n}}\right)^{k-1}.  \] Hence, $ r =
d(u),\; \forall u \in V$. That is $\h$, is $r$-regular.
\end{proof}
The sufficient condition of Theorem \ref{teo:reg} has already been demonstrated in \cite{Banerjee}.

Theorem \ref{TeoDd} below is a generalization for hypergraphs of Theorem 2.8
for graphs due to Cioab\u{a} and Gregory in \cite{Cioaba}.
\begin{Teo}\label{TeoDd}
	Let $\h$ be a connected hypergraph. If $(\rho, \x)$ is its  principal
eigenpair, then
\[\gamma(\h)\geq \max\left\lbrace  \left(
\frac{\Delta}{\rho}\right) ^{\frac{1}{k-1}}, \left(
\frac{\rho}{\delta}\right) ^{\frac{1}{k-1}}\right\rbrace\]
	Moreover
	\begin{enumerate}
\item If all the vertices of maximum and minimum degrees in $\h$ are
    contained only in edges of maximum cardinality, then equality occurs,
    if and only if, both statements are true
		\begin{enumerate}
			
        \item For each $u$ of maximum degree we have $x_u = x_{max}$ and
            $x_p = x_{min}$ whenever $u$ and $p$ are adjacent.
			
        \item For each $v$ of minimum degree we have $x_v = x_{min}$ and
            $x_q = x_{max}$ whenever $v$ and $q$  are adjacent.
		\end{enumerate}
		
\item If any vertex of maximum or minimum degree in $\h$ is contained in an
    edge which does not have maximum cardinality, then equality occurs, if
    and only if, $\h$ is regular.	
	\end{enumerate}
\end{Teo}
\begin{proof} Let $u,v \in V$ be vertices, such that $d(u) = \Delta$ and $d(v) = \delta$,
so we have
\[\rho x_{max}^{k-1}\geq \rho x_u^{k-1} =  \sum_{e\in
E_{[u]}}a(e)\sum_{\alpha \in S(e)_u}x^{\alpha-u}\geq \sum_{e\in E_{[u]}}x_{min}^{k-1}
= \Delta x_{min}^{k-1}  \Rightarrow \left(
\frac{x_{max}}{x_{min}}\right)^{k-1}\geq \frac{\Delta}{\rho}\]

\[\rho x_{min}^{k-1}\leq \rho x_v^{k-1} =  \sum_{e\in E_{[v]}}a(e)\sum_{\alpha \in S(e)_v}x^{\alpha-v}\leq \sum_{e\in E_{[v]}}x_{max}^{k-1} = \delta x_{max}^{k-1} \Rightarrow \left( \frac{x_{max}}{x_{min}}\right)^{k-1}\geq \frac{\rho}{\delta}\]

\begin{enumerate}
	\item Notice that
	\[\rho x_{max}^{k-1} = \rho x_u^{k-1}\; \Leftrightarrow\; x_u = x_{max}, \quad \sum_{e \in E_{[u]}}x^{e-u}=\Delta x_{min}^{k-1}\;\Leftrightarrow\;  x_p = x_{min}\; \forall p \in N(u)\]
	\[\rho x_{min}^{k-1} = \rho x_v^{k-1}\; \Leftrightarrow\; x_v = x_{min}, \quad \sum_{e \in E_{[v]}}x^{e-u}=\delta x_{max}^{k-1}\;\Leftrightarrow\;  x_q = x_{max}\; \forall q \in N(v)\]
	
	\item Suppose that there is a vertex $u$ of maximum degree contained in one edge that does not have maximum cardinality, the other case is analogous.
	\[\rho x_{max}^{k-1} = \rho x_u^{k-1}\; \Leftrightarrow\; x_u = x_{max},\]\[\sum_{e\in E_{[u]}}a(e)\!\!\!\sum_{\alpha \in S(e)_u}x^{\alpha-u}=\Delta x_{min}^{k-1}\;\Leftrightarrow\;  x_p = x_{min}\; \forall p \in N(u)\cup\{u\}\]
	That is, we would have $x_{max} = x_{min}$ and therefore equality is true, if and only if, $\h$ is regular.	
\end{enumerate}
\end{proof}

We observe that in Theorem \ref{TeoDd} the equality is true, for example, for
the star $S_n$.

\begin{Cor}\label{coroli}
	Let $\h$ be a hypergraph. If $ (\rho, \x) $ is its principal eigenpair, then
		\[\gamma(\h)\geq \left( \frac{\Delta}{\delta}\right)^{\frac{1}{2(k-1)}}\]
		Equality holds under the same conditions of Theorem \ref{TeoDd}.	
\end{Cor}
\begin{proof} We notice that $\left( \frac{\Delta}{\delta}\right)^{\frac{1}{2(k-1)}}$
is the geometric mean between $
\left(\frac{\Delta}{\rho}\right)^{\frac{1}{k-1}}$ and $ \left(
\frac{\rho}{\delta}\right)^{\frac{1}{k-1}}$, so
	\[\gamma(\h)\geq \max\left\lbrace  \left(\frac{\Delta}{\rho}\right)^{\frac{1}{k-1}}, \left(\frac{\rho}{\delta}\right)^{\frac{1}{k-1}}\right\rbrace \geq  \left( \frac{\Delta}{\delta}\right)^{\frac{1}{2(k-1)}}\]
\end{proof}

 \begin{Teo}
 	Let $\h$ be a connected hypergraph on $n$ vertices. If $(\rho, \x)$ is its
principal eigenpair, then
 	\[\sigma(\h) \geq \frac{\Delta^\frac{1}{2(k-1)} - \delta^\frac{1}{2(k-1)}}{\Delta^\frac{1}{2(k-1)}n^\frac{1}{k}}\]
 	Equality holds if and only if $\h$ is regular.
 \end{Teo}
\begin{proof} Firstly, we observe that
\[\frac{x_{max}}{x_{min}}\geq \left(
\frac{\Delta}{\delta}\right)^\frac{1}{2(k-1)} \Rightarrow x_{min} \leq\left(
\frac{\delta}{\Delta}\right)^\frac{1}{2(k-1)} x_{max}\]

Multiplying the inequality by $ -1 $ and adding $ x_ {max} $  to both sides,
we arrive at the following inequality
\[x_{max} - x_{min} \geq \left(1 -
\left( \frac{\delta}{\Delta}\right)^\frac{1}{2(k-1)}\right)x_{max} =
\frac{\Delta^\frac{1}{2(k-1)} -
\delta^\frac{1}{2(k-1)}}{\Delta^\frac{1}{2(k-1)}}x_{max} \] Note that
$x_{max} \geq \frac{1}{\sqrt[k]{n}}$, so we conclude that
 \[x_{max} - x_{min} \geq \frac{\Delta^\frac{1}{2(k-1)} - \delta^\frac{1}{2(k-1)}}{\Delta^\frac{1}{2(k-1)}n^\frac{1}{k}}\]
 \vspace{0.2cm}

 Equality occurs, if and only if, $x_{max} = \frac{1}{\sqrt[k]{n}}$, that is whenever $\h$ is regular.
\end{proof}

\section{Bounds for the principal eigenvector entries}\label{sec:bounds}
In this section we present some results on the extreme entries of the principal eigenvector of a general hypergraph, relating it with important classical parameters.

The Theorem \ref{tminmax} is a generalization for hypergraphs of Lemma 3.3 for graphs given by Cioab\u{a} and Gregory in \cite{Cioaba}.

\begin{Teo}\label{tminmax}	Let $\h$ be a hypergraph on $n$
vertices. If $ (\rho, \x) $ is its principal eigenpair, then	
\begin{itemize}	
\item[(a)] $x_{max} \geq  \frac{1}{\sqrt[k]{\left(
    \frac{\delta}{\Delta}\right)^{\frac{k}{2(k-1)}} +n-1}}$. For $k \geq
    3$, the equality holds if and only if $\h$ is regular.
		
\item[(b)] $x_{min} \leq  \frac{1}{\sqrt[k]{\left(
    \frac{\Delta}{\delta}\right)^{\frac{k}{2(k-1)}} +n-1}}$.  For $k \geq
    3$, the equality holds if and only if $\h$ is regular.
	\end{itemize}
\end{Teo}	
\begin{proof}

To prove part (a), we observe that \[1 = \sum_{u \in V} x_v^k \leq x_{min}^k
+    (n-1)x_{max}^k =(\gamma^{-k} +n -1)x_{max}^k \leq
    \left(\left(\frac{\delta}{\Delta} \right)^\frac{k}{2(k-1)}+n-1
    \right)x_{max}^k. \]

Now note that the equality occurs if, and only if, $n-1$ entries of $\x$ are equal $x_{max}$ and one entry is equals $x_{min}$. Observe yet that equality occurs only if the equality on  Theorem \ref{TeoDd} occurs as well. Under these conditions we can assume that all vertices of maximum or minimum degree are contained only in edges of maximum cardinality -- otherwise, the hypergraph would be regular. I.e., for each $u \in V$ such that $d(u) = \Delta$ must be true $x_q = x_{min}$ for every neighbor $q$ of $u$. Since $k\geq3$, then $u$ must have at least two neighbors, 
but only one can take on a value other than $x_{max}$, therefore $x_{min}
= x_{max}$, that is $\h$ is regular.

Similarly, we prove part (b)
\end{proof}

Theorem \ref{txm} below is a generalization for hypergraphs of Theorem 3.4
for graphs due to Cioab\u{a} and Gregory \cite{Cioaba}.
\begin{Teo}\label{txm}
Let $\h$ be a hypergraph. If $ (\rho, \x) $ is its principal eigenpair, then	

\begin{itemize}
	\item[(a)] $\displaystyle x_{max}\geq \sqrt[k]{\frac{\rho}{nd(\h)}}$. Equality holds if and only if
	$\h$ is regular.
	\item[(b)] $\displaystyle x_{max}\geq \frac{\rho^{\frac{1}{k-1}}}{\left( \displaystyle \sum_{v \in V}d(v)^\frac{k}{k-1}\right)^\frac{1}{k} }$. Equality holds if and only if $\h$ is regular.
\end{itemize}
\end{Teo}	
\begin{proof}
To prove the part (a), observe that, for each $u \in V$ must be true
\[\rho x_u^{k-1} = \sum_{e\in E_{(u)}}a(e)\sum_{\alpha \in S(e)_u}x^{\alpha-u} \leq \sum_{e \in E_{(u)}}x_{max}^{k-1} = d(u)x_{max}^{k-1}\]

So it is worth $\rho x_u^k \leq d(u)x_{max}^k$.  Summing over the set of vertices, we have
\[\rho \sum_{u \in V}x_u^{k} \leq \sum_{u \in V}d(u)x_{max}^k\quad \Rightarrow \quad \rho \leq nd(\h)x_{max}^k \quad \Rightarrow \quad  x_{max}\geq \sqrt[k]{\frac{\rho}{nd(\h)}}\]

Note that the equality occurs, if and only if, $ x_u = x_{max} $ for all $ u \in V $. That is, equality occurs only when $ \h $ is regular.

For the part (b), notice that $\rho x_u^{k-1} \leq d(u)x_{max}^{k-1}$, hence $\rho^\frac{k}{k-1} x_u^{k} \leq d(u)^\frac{k}{k-1}x_{max}^{k}$. Adding on the set of vertices, we have \[\rho^\frac{k}{k-1}\sum_{u \in V}
x_u^{k} \leq  \sum_{u \in V}d(u)^\frac{k}{k-1}x_{max}^{k}\quad \Rightarrow
\quad \rho^\frac{k}{k-1} \leq x_{max}^k\left( \sum_{u \in
V}d(u)^\frac{k}{k-1}\right).    \]

For the same reason of the part (a), the equality holds only when $\h$ is regular.
\end{proof}

Now we will prove Theorem \ref{teoxmin1}, which we state again for easy
reference.

\noindent\textbf{Theorem \ref{teoxmin1}.}	\textit{Let $ \h $ be a connected hypergraph on $n$ vertices. If $(\rho, \x)$ is its
	principal eigenpair, then}

	\[x_{min} \leq \sqrt[k]{\frac{\delta}{\rho+\delta(n-1)}}.\]

\textit{The equality holds if and only if there is a vertex $v$ such that, for all $w
\in V-\{v\}$, we have  $x_w = x_{min}$ and $x_v =
\sqrt[k]{\frac{\rho}{\delta}}x_{min}$.}

\begin{proof} We notice that
\[\rho x_i^{k} = \sum_{e \in E_{[i]}}a(e)\sum_{\alpha
\in S(e)_i}x^{\alpha} \leq \sum_{e \in E_{[i]}}a(e)\sum_{\alpha \in
S(e)_i}x^{max} = d(i) x_{max}^k,\,\forall i \in V(\h).\]

Therefore
\begin{equation}\label{EqmM}
\rho x_{min}^k \leq \delta x_{max}^k
\end{equation}
Let us choose a vertex $v$ such that $x_v = x_{max}$, so the following
inequality holds
\begin{equation}\label{EqmM2}
\delta(n-1)x_{min}^k \leq \delta\sum_{i \in V-\{v\}}x_i^k
\end{equation}

Adding the inequalities \ref{EqmM} and \ref{EqmM2} we conclude
that
\[x_{min}^k(\rho+\delta(n-1)) \leq \delta\sum_{i \in V}x_i^k = \delta
\quad \Rightarrow \quad x_{min} \leq
\sqrt[k]{\frac{\delta}{\rho+\delta(n-1)}}\]

To finish the proof, we just notice that

\[\delta(n-1)x_{min}^k =
\delta\sum_{i \in V-\{v\}}x_i^k\quad \Leftrightarrow \quad x_w = x_{min}\;
\forall\; w \neq v.\]

Further \[\rho x_{min}^k = \delta x_{max}^k \quad
\Leftrightarrow \quad x_{v} = \sqrt[k]{\frac{\rho}{\delta}}x_{min}\]
\end{proof}
Note that the equality in Theorem \ref{teoxmin1} is true, for example,
for regular hypergraphs or for the star $S_n$.
\begin{Obs}
	The equality in Theorem \ref{teoxmin1} is sharper
	than the Theorem \ref{tminmax}, when
	$$\rho \geq \displaystyle\sqrt[2(k-1)]{\Delta^k\delta^{k-1}}$$
\end{Obs}
	For completeness, we will prove this statement. Indeed,
	
	$$\rho \geq \displaystyle\sqrt[2(k-1)]{\Delta^k\delta^{k-1}} \quad\Rightarrow\quad \frac{\rho}{\delta} \geq \displaystyle\left( \frac{\Delta}{\delta}\right)^\frac{k}{2(k-1)} \quad\Rightarrow\quad  \frac{1}{\frac{\rho}{\delta} +n-1}\leq \frac{1}{\left( \frac{\Delta}{\delta}\right)^\frac{k}{2(k-1)} +n-1}$$
	Thus we conclude that
	
	$$\sqrt[k]{\frac{\delta}{\rho+\delta(n-1)}} \leq \frac{1}{\sqrt[k]{\left(
			\frac{\delta}{\Delta}\right)^{\frac{k}{2(k-1)}} +n-1}}$$

\section{Measures for centering and regularity of edges}\label{sec:edges}

In this section we will study the parameters  $x^{max}$ and $x^{min}$ that
are still little explored, even for graphs. Just as $x_{max}$ and $x_{min}$
can be used to determine the most central and peripheral vertices, we believe
that the values $x^{max}$ and $x^{min}$ have a similar role for the edges of
a uniform hypergraph.

\begin{Def}
	Let $\h$ be a $k$-graph and $\x=(x_v)$ its principal eigenvector.
	\[x^{min} = \min_{e \in E(\h)}\{x^e\}, \quad x^{max} = \max_{e \in E(\h)}\{x^e\},\quad \Gamma(\h) = \frac{x^{max}}{x^{min}}\]			
\end{Def}

\begin{Teo}\label{lemapeso}	
Let $\h$ be a connected $k$-graph. If $(\rho, \x)$ is its principal
eigenpair, then
	\begin{itemize}
		\item[(a)] $\frac{\delta}{\rho}x^{min} \leq x_{min}^k \leq x^{min}$
		\item[(b)] $x^{max} \leq x_{max}^k \leq \frac{\Delta}{\rho}x^{max}$
	\end{itemize}
	If $\h$ is regular then the equalities hold.
\end{Teo}
\begin{proof}
  For part (a) notice that for every $v \in V$ we have $x_{min} \leq x_v \leq
    x_{max}$, thus given an edge $e = \{v_1,\cdots, v_k\} \in E$, we have
    \[x^e = x_{v_1}\cdots x_{v_k} \geq  x_{min}\cdots
    x_{min} = x_{min}^k \quad \forall e \in E  \quad \Rightarrow \quad
    x^{min}\geq x_{min}^k.\]

    Let $u \in V$ such that $x_u = x_{min}$, so
\[\rho x_{min}^{k} = \sum_{e \in E_{[u]}}x^{e} \geq \sum_{e \in E_{[u]}}x^{min} \geq
\delta x^{min}.\]

Similarly we prove the part (b).

If $\h$ is regular then $x_{min} = x_{max}$ and $\delta = \rho = \Delta$,
therefore the equality holds.
\end{proof}

\begin{Teo}\label{Teominrhormax}
Let $ \h $ be a connected $k$-graph on $n$ vertices and $m$ edges, if $(\rho, \x)$ is its
principal eigenpair, then
		\[x^{min} \leq \frac{\rho(\h)}{km} \leq x^{max}\]	
	
	If $\h$ is regular then the equality holds.
\end{Teo}
\begin{proof}

To prove the first inequality, just note that

\[\rho(\h) = k\sum_{e \in E}x^e \geq k\sum_{e \in E}x^{min} = kmx^{min} \]

The other inequality is analogous.

If $\h$ is regular then $x^{min} = \frac{1}{n} = x^{max}$ and $\rho = \frac{km}{n}$,
therefore the equality holds.
\end{proof}

One consequence of this theorem is that $\sqrt[k]{\frac{\rho}{km}} \leq x_{max}$, 	because $x^{max} \leq x_{max}^k$. It is also possible to obtain the following inequality  $x_{min}\leq \sqrt[k]{\frac{\rho}{km}}$,
but it is not very interesting because $x_{min}\leq\frac{1}{\sqrt[k]{n}}\leq \sqrt[k]{\frac{\rho}{km}}$.

\begin{Teo}
	Let $\h$ be a connected $k$-graph, if $(\rho, \x)$ is its  principal
eigenpair, then \[\sqrt[k]{\Gamma(\h)}\leq\gamma(\h)\leq
\sqrt[k]{\frac{\Delta}{\delta}\Gamma(\h)}.\]	
	If $\h$ is regular then both equalities hold.
\end{Teo}
\begin{proof} These inequalities follow from Theorem \ref{lemapeso}:
\[x^{max} \leq (x_{max})^k, \quad x^{min} \geq (x_{min})^k
\quad\Rightarrow\quad \Gamma(\h)\leq\gamma(\h)^k.\]

\[x_{max}^{k}\leq \frac{\Delta}{\rho} x^{max}, \quad x_{min}^{k} \geq \frac{\delta}{\rho} x^{min}\quad\Rightarrow\quad \gamma(\h)^k\leq \frac{\Delta}{\delta}\Gamma(\h)\]

We observed that if $\h$ is regular then $\gamma = 1$, $\Gamma = 1$  and
$\Delta = \delta$, therefore equality holds.
\end{proof}

Obviously if $\h$ is regular then $\gamma = 1\; \Rightarrow\; x_{max} =
x_{min} \; \Rightarrow\; x^{max} = x^{min} \; \Rightarrow\; \Gamma =1$. But
it is not true that if $\Gamma = 1$ then $\h$ is regular. Thus the following
questions naturally arise. For which hypergraphs $\h$, we have $\Gamma(\h) =
1 $? What is the meaning of the parameter $\Gamma$? We know that $\gamma$
measures the regularity of vertices of the hypergraph, would $\Gamma$ have a
similar meaning for the edges?

We believe that these questions are answered by Theorem \ref{teo:caracG}
which will be proved below, which we state here again for easy reference.

\vspace{0.4cm}\noindent\textbf{Theorema \ref{teo:caracG}.}	\textit{Let $\h$ a $k$-graph. $\Gamma(\h) = 1$, if and only if, for each edge the product of the degrees of its vertices is constant.}

\begin{proof} If $\Gamma(\h) = 1$ then $x^{min} = x^{max}$, so by Theorem
\ref{Teominrhormax} we have $x^e = \frac{\rho}{km}$, for all $e\in E$. For
each $u \in V$ we have that\[\rho x_u^k = \sum_{e \in E_{[u]}}x^e =
d(u)\frac{\rho}{km}\quad \Rightarrow\quad x_u = \sqrt[k]{\frac{d(u)}{km}}.\]
Therefore \[\frac{\rho}{km} = x^e = \frac{\sqrt[k]{d(v_1)\cdots d(v_k)}}{km}
\quad \Rightarrow\quad \rho^k = d(v_1)\cdots d(v_k)\quad \forall\;
e=\{v_1,\ldots,v_k\} \in E.\]

That is, for any edge, the product of the degrees of its vertices is always $\rho^k$.

\vspace{0.2cm}

Conversely, if $d(v_1)\cdots d(v_k) = D$ for all
$e=\{v_1,\ldots,v_k\} \in E$, we define the vector $\x$ by $x_u =
\sqrt[k]{\frac{d(u)}{km}}$ for all $u \in V$, and then
\[\sum_{e \in E_{[u]}}x^{e}
= \sum_{e \in E_{[u]}} \frac{\sqrt[k]{d(v_1)\cdots d(v_k)}}{km} = \sum_{e \in
E_{[u]}} \frac{\sqrt[k]{D}}{km} =  d(u)\frac{\sqrt[k]{D}}{km} = \sqrt[k]{D} x_u^k.\]

Hence $(\Ah\x)_u = \sqrt[k]{D} x_u^{k-1}$ for all $u \in V$. That is,
$(\sqrt[k]{D},\x)$ is an eigenpair of $\h$ and by Theorem \ref{Perron} we
know that $\x$ is the principal eigenvector of $\h$, so $x^{max} =
\frac{\sqrt[k]{D}}{km} = x^{min}$ and therefore $\Gamma(\h) = 1$.
\end{proof}

Under these conditions, it is reasonable to say that the parameter
$\Gamma$ measures the balance of the  distribution of vertex degrees at the
edges in the hypergraph, since if $\Gamma$ is greater than 1 then the product
of the vertices of each edge is not constant.

\begin{Def}
Let $\h=(V,E)$ be an $k$-graph, let $s \geq 1$ and $r \geq ks$ two integers.
We define a generalized power hypergraph of $\h$ as the $r$-graph $\h^r_s$,
obtained by replacing each vertex $v_i \in V(\h)$ by a set with $s$ vertices
$\varsigma_{v}=\{v_{1}, \ldots, v_{s}\}$ and adding a new set with $r-k$ vertices of
degree one $\varsigma_e=\{v^1_e, \ldots,v^{r-k}_e\}$ on each edge $e \in E(\h)$. More
precisely, the sets of vertices and edges of $\h^r_s$ are
$$V(\h^r_s)=\left( \bigcup_{v\in V} \varsigma_v\right) \cup \left( \bigcup_{e\in E} \varsigma_e\right)\;\; \emph{and}\;\; E(\h^r_s)=\{\varsigma_e\cup \varsigma_{v_1} \cup \cdots \cup \varsigma_{v_k} \colon e=\{v_1,\ldots, v_k\} \in E\}$$
	
\end{Def}

\begin{Exe}
Two examples of families of uniform hypergraphs that verify the conditions of
the Theorem \ref{teo:caracG} are
\begin{itemize}
	\item[(a)] The generalized power hypergraph $\h^r_s$ where $\h$ is a regular
	$k$-uniform hypergraph and $r \geq ks$.
	\item[(b)] The generalized power hypergraph $(S_n)^r_s$  where $S_n$ is the graph
	star with $n$ vertices and $r \geq 2s$.
\end{itemize}	
	
\end{Exe}

\section*{Acknowledgments}
This work is part of doctoral studies of K. Cardoso under the supervision of
V. Trevisan. K. Cardoso is grateful for the support given by Intituto Federal
do Rio Grande do Sul (IFRS), Campus Feliz. V. Trevisan acknowledges partial
support of CNPq grants 409746/2016-9 and 303334/2016-9, CAPES (Proj.
MATHAMSUD 18-MATH-01) and FAPERGS (Proj.\ PqG 17/2551-0001).


\end{document}